\documentclass[a4paper,10pt]{amsart}
 \usepackage{amsfonts,mathrsfs}
 \usepackage{amsthm}
 \usepackage{amssymb}
 \usepackage{arydshln}
 \usepackage{tikz-cd}
 \usepackage{tikz}
 \usepackage{cleveref}
 \usepackage{url}
 \usepackage{fourier}
 \usepackage{comment}
 
\theoremstyle{definition}
\newtheorem{defin}{Definition}[section]
\newtheorem{ex}[defin]{Example}
\newtheorem{rem}[defin]{Remark}

\theoremstyle{plain}
\newtheorem{lem}[defin]{Lemma}
\newtheorem{prop}[defin]{Proposition}
\newtheorem{thm}[defin]{Theorem}

\newtheorem*{thm*}{Theorem}
\newtheorem{cor}[defin]{Corollary}
\newtheorem*{cor*}{Corollary}

\newtheorem*{quest*}{Question}
\newtheorem{conj}[defin]{Conjecture}
\newtheorem*{conj*}{Conjecture}

\usepackage{mathrsfs}
\newcommand{\sing}{\operatorname{sing}}

\newcommand{\rk}{\operatorname{rk}}

\newcommand{\cC}{{\mathcal C}}

\newcommand{\cF}{{\mathcal F}}

\newcommand{\cL}{{\mathcal L}}

\newcommand{\cO}{{\mathcal O}}
\newcommand{\cP}{{\mathcal P}}

\newcommand{\cV}{{\mathcal V}}

\newcommand{\C}{{\mathbb C}}
\newcommand{\R}{{\mathbb R}}
\newcommand{\pp}{\mathbb{P}}

\title[The Geometry of Polycons and Wachspress' Conjecture]{The Geometry of Polycons and a  Counterexample to Wachspress' Conjecture}

\author{Clemens Brüser}
\address{Technische Universit\"at Dresden, Germany} 
\email{c.brueser.acad@posteo.de}

\thanks{The author has been partly supported by the DFG grant 502861109.}
\begin{document}

\subjclass[2010]{Primary: 14H10, 14N30, 14P10, 14P25}

\begin{abstract}
 Polycons, initially introduced by Wachspress in 1975 as a tool in finite element methods, are generalizations of polygons in that they allow conic boundary components. We are interested in the adjoint curve of a given polycon, i.e. the unique curve of minimal degree vanishing in the so-called residual arrangement. It was conjectured by Wachspress that under some regularity assumptions this curve does not vanish in the interior of its defining polycon. However, until recently the only class of polycons for which this was proven were convex polygons. We present a polycon bounded by three conics that constitutes a counterexample to Wachspress' conjecture.
 
 The origin of this counterexample reveals some beautiful geometry of polycons. Replacing one degree two boundary component of a polycon with a line produces a new polycon. We show that the adjoint of the latter is a contact curve to the adjoint of the former. This naturally leads to the consideration of symmetric linear determinantal representations of adjoints, which lets us explicitly describe the fibers of the adjoint map in the case of polycons bounded by three conics. As a corollary we prove that generically the adjoint of a polycon bounded by three conics is smooth.
\end{abstract}
\maketitle

\noindent \textbf{Acknowledgements.} I extend my gratitude to Mario Kummer, Eugene Wachspress, and the authors of \cite{KohnEtAl2024AdjCurves} for bringing this topic to my attention and for valuable discussions and suggestions that occurred throughout the preparation of this article.

\noindent \textbf{Declaration on the Use of AI Tools.} I declare that during preparation of this manuscript AI search assistants were used for the purpose of writing \texttt{python} code for the visualization of results. The output was reviewed and edited as needed. The full responsibility for the content of this article lies with me.

\section{Introduction}

\emph{Polycons} are a generalization of polygons originally introduced by \cite{Wachspress1975Adjoints} (reprinted in \cite[Part I]{Wachspress2016Adjoints}) as a tool in finite element methods. Contrary to polygons, a polycon allows for boundary components that are segments of conics.

For a polycon the set of singular points of its (algebraic) boundary that are not vertices is called the \emph{residual arrangement}, and under some regularity conditions it has been shown that there exists a unique curve of minimal and predetermined degree passing through the residual arrangement, called the \emph{adjoint} curve.

One of the main questions in the study of adjoints is whether there exist points in the interior of a real polycon, in which the adjoint vanishes. This problem is commonly referred to as \emph{Wachspress' Conjecture} and was open even in one of its simplest instances: for polycons bounded by three conics.

Attention has also recently been drawn to the study of \emph{adjoint maps}, which assign to a (potentially complex) polycon with boundary components of fixed degree its adjoint (\cite{AgostiniPlaumannSinnWesner2024AdjPolypol} and \cite[\S4]{KohnEtAl2024AdjCurves}). While these papers focus on finite adjoint maps, we are interested in continuous parametrizations of polycons giving rise to the same adjoint.

\subsection{Outline of the Article}

This article comprises several topics. Our main result is a counterexample to Wachspress' Conjecture (\Cref{ex: counterexample}, \Cref{fig:Wachspress_CounterExample}, and \Cref{thm: WachspressWrong}). We state this example as early as possible, just after recalling the basic notions of polycons and their adjoints (\Cref{sec:PrelimPolycon}).

During the process of writing this paper, a link between polycons and the theory of symmetric linear determinantal representations was found. After introducing some necessary background in \Cref{sec:PrelimAG} we present two results related to this connection in \Cref{sec:mainResult}. The first is \Cref{thm: adj_recursion}: given a polycon $\cP$ with adjoint $A$, then by replacing a conic boundary component of $\cP$ with a line, one obtains a new polycon $\cP'$ with adjoint $A'$, such that $A'$ is a contact curve to $A$ with predetermined contact divisor.

The second result is \Cref{prop: fiber222}: given a fixed cubic curve, we determine a parametrization of all polycons bounded by three conics with the same adjoint curve. These families of polycons correspond to families of symmetric linear determinantal representations of the cubic. Phrased dfferently, we describe the fiber of the adjoint map for polycons bounded by three conics. This will allow us to conclude in passing that the adjoint of a polycon bounded by three conics is generically smooth.

\section{Polycons, Adjoints, and Wachspress' Conjecture} \label{sec:PrelimPolycon}

\subsection{Polycons and their Adjoints}

We begin by introducing the notion of a polycon. We take an approach that differs slightly from recent work by \cite{KohnEtAl2024AdjCurves} in that we restrict considerations to boundary components of degree at most two. At the same time this allows us to take a less restrictive approach to the regularity of the boundary components. For reference, the content of this section relies mainly on \cite{Wachspress2016Adjoints} and \cite{KohnEtAl2024AdjCurves}.

As a standing assumption we assume that in any family of objects $O_{12}, O_{23}, \dots, O_{n1}$ the labels are cyclically repeating, i.e. we set $O_{n,n+1} = O_{n1}$. We make an analogous convention for objects labeled just by $1, \dots, n$, i.e. $O_{n+1} = O_1$.

\begin{defin}[Polycon]
 In the following assume $n \geq 2$. A \emph{polycon} $\cP$ consists of
 \begin{enumerate}
     \item a tuple $(C_i)_{i=1}^n$ of reduced curves of degree one or two, called \emph{boundary components},
     \item a tuple $(v_{i,i+1})_{i=1}^n$ of points in $\pp^2$, called \emph{explicit vertices}, such that $C_i$ and $C_{i+1}$ intersect transversally in $v_{i,i+1}$ and no other boundary component contains $v_{i,i+1}$.
 \end{enumerate}
 A polycon is called \emph{real}, if all input data is real, and additionally
 \begin{enumerate}
     \item there is a choice of segments $(\sigma(C_i))_{i=1}^n$ on the boundary components with end points $v_{i-1,i}$ and $v_{i,i+1}$, called \emph{sides},
     \item there exists a closed semi-algebraic set $\cP_{\geq0}$, the interior of which is a union of simply connected sets, and whose boundary is the union of the sides of $\cP$.
 \end{enumerate}
\end{defin}

If a boundary component $C_i$ is a singular conic, then we call its unique singular point an \emph{implicit vertex} of $\cP$. We denote the set of all vertices of $\cP$ by $V(\cP)$. The curve $C_\cP := \bigcup_i C_i$ is called the \emph{algebraic boundary} of $\cP$ and we call $\deg(\cP) := \deg(C_\cP)$ the \emph{degree} of $\cP$. If working with real polycons, then the union $\partial \cP := \bigcup_i \sigma(C_i)$ of all sides of $\cP$ is called the \emph{Euclidean} boundary of $\cP$. Finally, we call the open semi-algebraic set $\cP_{>0} := \cP_{\geq0} \setminus \partial\cP$ the \emph{interior} of $\cP$.

\begin{figure}[!ht]
    \centering
    \includegraphics[width=0.4\textwidth]{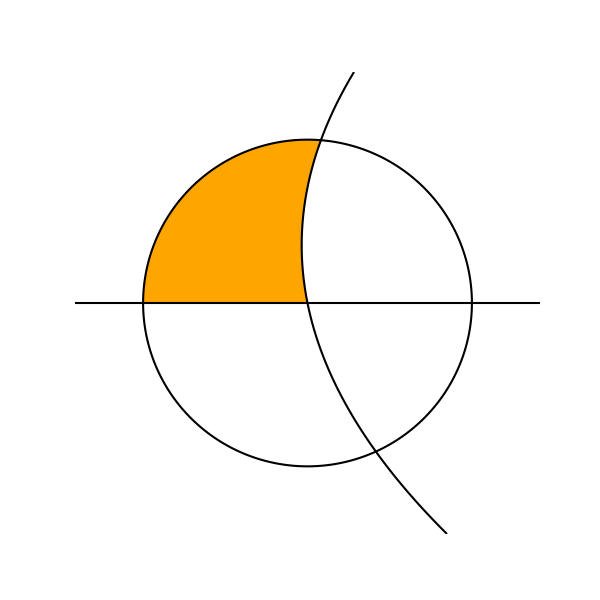}
    \includegraphics[width=0.4\textwidth]{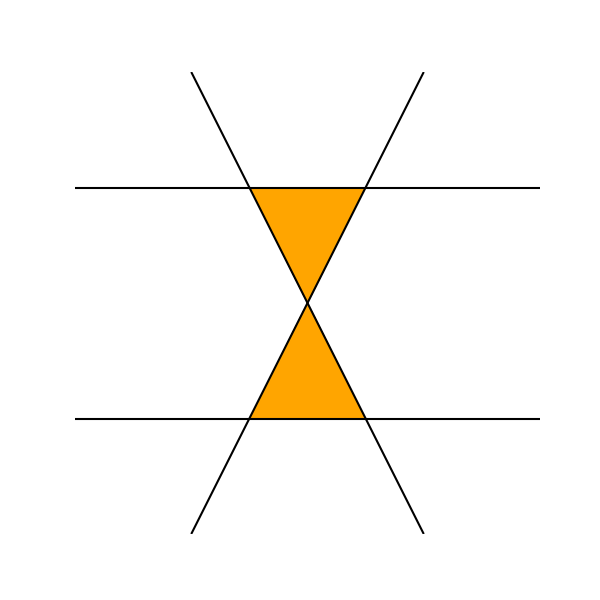}
    \caption{Examples of a regular (left) and an irregular (right) real polycon (shaded in orange). Both polycons are nodal.}
    \label{fig:polycon_exmpl}
\end{figure}

The constellation of vertices and boundary components of a polycon can display pathological behaviour. The absence of these peculiarities leads to the notions of nodal and regular polycons (see \Cref{fig:polycon_exmpl}).

\begin{defin}[Nodal and regular polycons]
 A polycon $\cP$ is called \emph{nodal}, if all singularities of its algebraic boundary $C_\cP$ are nodes. A real polycon is called \emph{regular}, if the only singular points of $C_\cP$ on $\partial \cP$ are the vertices of $\cP$ and $C_\cP$ does not intersect $\cP_{>0}$.
\end{defin}

\begin{defin}[Residual Arrangement]
 Let $\cP$ be a polycon. The set $R(\cP) := \sing(C_\cP) \setminus V(\cP)$ is called the \emph{residual arrangement} of $\cP$.
\end{defin}

The adjoint curve of $\cP$ aims to be the (unique) curve of minimal degree vanishing in $R(\cP)$. If $\cP$ is nodal, then this may be taken as its definition (see \Cref{fig:adjoint_exmpl} for an illustration). However, we do not want to impose this restriction. In order to still obtain a unique object, one has to impose additional vanishing conditions on the adjoint curve in the residual arrangement.

For this, one considers the desingularization of the algebraic boundary $C_\cP$ in the residual arrangement. This yields a morphism $\nu \colon Z \to C_\cP$, where $Z$ is the strict transform of $C_\cP$ under a sequence of blow-ups of $\pp^2$. By the ideal sheaf $\cC_\nu$ on $C_\cP$ we denote the conductor associated with the morphism $\nu$.

\begin{defin}[Adjoint] \label{def: adjoint}
 An \emph{adjoint polynomial} $\alpha_\cP$ of a degree $d$ polycon $\cP$ is a global section of the sheaf $\cO_{\pp^2}(d-3) \otimes \iota_* \cC_\nu$ where $\iota \colon C_\cP \to \pp^2$ is the inclusion morphism. The vanishing locus $A_\cP$ of $\alpha_\cP$ is called an \emph{adjoint curve} of $\cP$.
\end{defin}

\begin{rem} \label{rem: AdjUnique_pathologies}
 The conditions imposed on an adjoint curve $A_\cP$ of $\cP$ by the conductor $\cC_\nu$ can be translated into explicit vanishing conditions, which must be satisfied by $A_\cP$ in the residual arrangement $R(\cP)$ (\cite[\S 5.1]{Wachspress2016Adjoints}). These can also be motivated by considering limits of nodal polycons. We give the two most common examples:
 \begin{enumerate}
     \item two boundary components tangent in a residual point impose that the adjoint's tangent direction in this point is the same as well.

     \item if at least three boundary components meet in a residual point, then the adjoint vanishes with multiplicity two in this point.
 \end{enumerate}
 One can furthermore relax the condition that each vertex of $\cP$ is only contained in only two boundary components, which intersect transversally. Again, two common examples are given below (see also \cite[Example 5.4]{Wachspress2016Adjoints}):
 \begin{enumerate}
     \item if $C_i$ and $C_{i+1}$ are tangent in $v_{i,i+1}$, then the adjoint vanishes in $v_{i,i+1}$.

     \item if there exists $j \neq i,i+1$ with $v_{i,i+1} \in C_j$, and if any two of the curves containing $v_{i,i+1}$ intersect transversally in this vertex, then the adjoint vanishes in $v_{i,i+1}$ with tangent direction equal to that of $C_j$.
 \end{enumerate}
\end{rem}

\begin{prop}[{\cite[Proposition 2.2]{KohnEtAl2024AdjCurves}}]\label{prop:uniqueAdj}
 For every polycon $\cP$ there exists a unique adjoint curve $A_\cP$. Furthermore $A_\cP$ and the algebraic boundary $C_\cP$ have no irreducible components in common.
\end{prop}

\begin{figure}[!ht]
    \centering
    \includegraphics[width=0.5\textwidth, trim = {0 3.5cm 0 3.5cm}, clip]{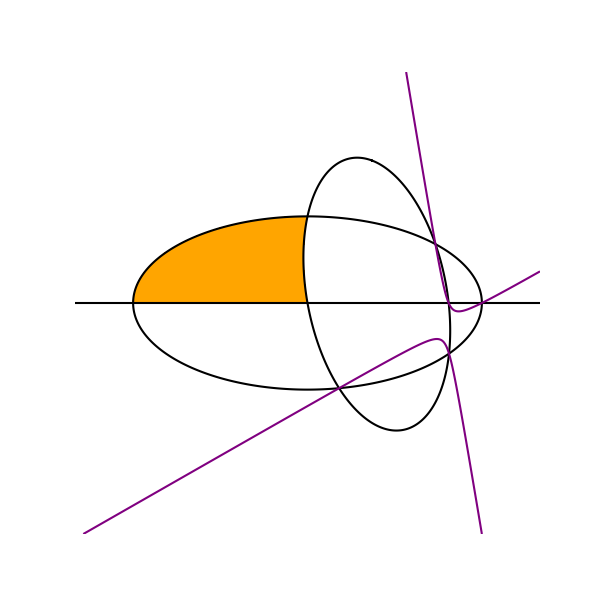}
    \caption{A regular degree five polycon (orange) and its conic adjoint (purple).}
    \label{fig:adjoint_exmpl}
\end{figure}

\begin{rem}
 An algorithm (the GADJ algorithm) for the robust computation of the adjoint curve has been described in \cite{DasguptaWachspress2008AdjointAlg}.
\end{rem}

We proceed to outline the fields of interest for this paper.

\subsection{Wachspress' Conjecture}

Since the inception of the study of adjoints an open problem had been Wachspress' Conjecture, which is a statement on the real topology of the adjoint curve relative to the interior of its defining real polycon.

\begin{conj}[{Wachspress' Conjecture, \cite[p.~396]{Wachspress1980AdjConjecture}, \cite[Conjecture 3.4]{KohnEtAl2024AdjCurves}}] \label{conj: Wachspress}
 If $\cP$ is a regular real polycon with adjoint curve $A$, then $A(\R)$ does not intersect $\cP_{\geq0}$.
\end{conj}

Since its formulation, progress on this conjecture had been sparse. It had been answered in the positive for convex polygons in \cite[p.~96, Example 5.1]{Wachspress2016Adjoints}, and a complete topological description of the adjoint has in this case been given in \cite[\S3.2]{KohnEtAl2024AdjCurves}. Furthermore, it is known that the adjoint does not intersect the boundary $\partial\cP$ of $\cP$. The following lemma is a special case of \cite[Lemma 3.5]{KohnEtAl2024AdjCurves}.

\begin{lem}[{\cite[Lemma 3.5]{KohnEtAl2024AdjCurves}}]\label{lem:AdjOffBoundary}
 If $\cP$ is a nodal polycon, then $A_\cP$ intersects $C_\cP$ exactly in $R(\cP)$, and the intersection multiplicity in each residual point equals 1. In particular $A$ is smooth in $R(\cP)$, and if $\cP$ is regular, then $A$ does not intersect $\partial\cP$.
\end{lem}

Already, if $\cP$ is bounded by three conics, then the answer to \Cref{conj: Wachspress} was open and only partial results could be established (\cite{WachspressJr2020AdjConjecture}, \cite[\S3.3]{KohnEtAl2024AdjCurves}). The authors of \cite{KohnEtAl2024AdjCurves} enumerate all (up to equivalence) possible arrangements of three ellipses bounding a real polycon of degree six. Then they prove on a case-by-case basis that in most cases Wachspress' conjecture holds. The cases, for which they do not present a proof, essentially boil down to the unpredictability whether the adjoint curve has an oval, and where this oval might lie relative to the interior of its defining polycon.

\begin{ex}[Counterexample to \Cref{conj: Wachspress}] \label{ex: counterexample}
 Consider the real polycon with vertices
 \begin{align*}
     v_{12} = (9,-6), && v_{23} = (0,0), && v_{31} = (-3,-6)
 \end{align*}
 and boundary components being the zero sets of the polynomials
 \begin{equation*}
 \begin{split}
     c_1(x,y) &= 20x^2+27y^2-120x+108y-864, \\
     c_2(x,y) &= 80x^2+102xy+57y^2-400x-96y, \\
     c_3(x,y) &= 32x^2-26xy+9y^2-96x+72y.
 \end{split}
 \end{equation*}
 Here, we work in a fixed affine chart. The adjoint of this polycon is the zero set of the polynomial
 \begin{equation*}
 \begin{split}
     \alpha(x,y) &=
     3440x^3-8400x^2y-762xy^2+1971y^3\\
     &+20720x^2+51168xy-1620y^2-193248x-96336y+342144.
 \end{split}
 \end{equation*}
 In \Cref{fig:Wachspress_CounterExample} one can observe that the adjoint curve has an oval, which lies in the interior of the polycon. In \Cref{thm: WachspressWrong} below we formally prove that this constitutes a counterexample to \Cref{conj: Wachspress}. Motivation for this example is given in \Cref{sec:SolutionGenesis}.
\end{ex}

\begin{figure}[!ht]
    \centering
    \includegraphics[width=1\textwidth]{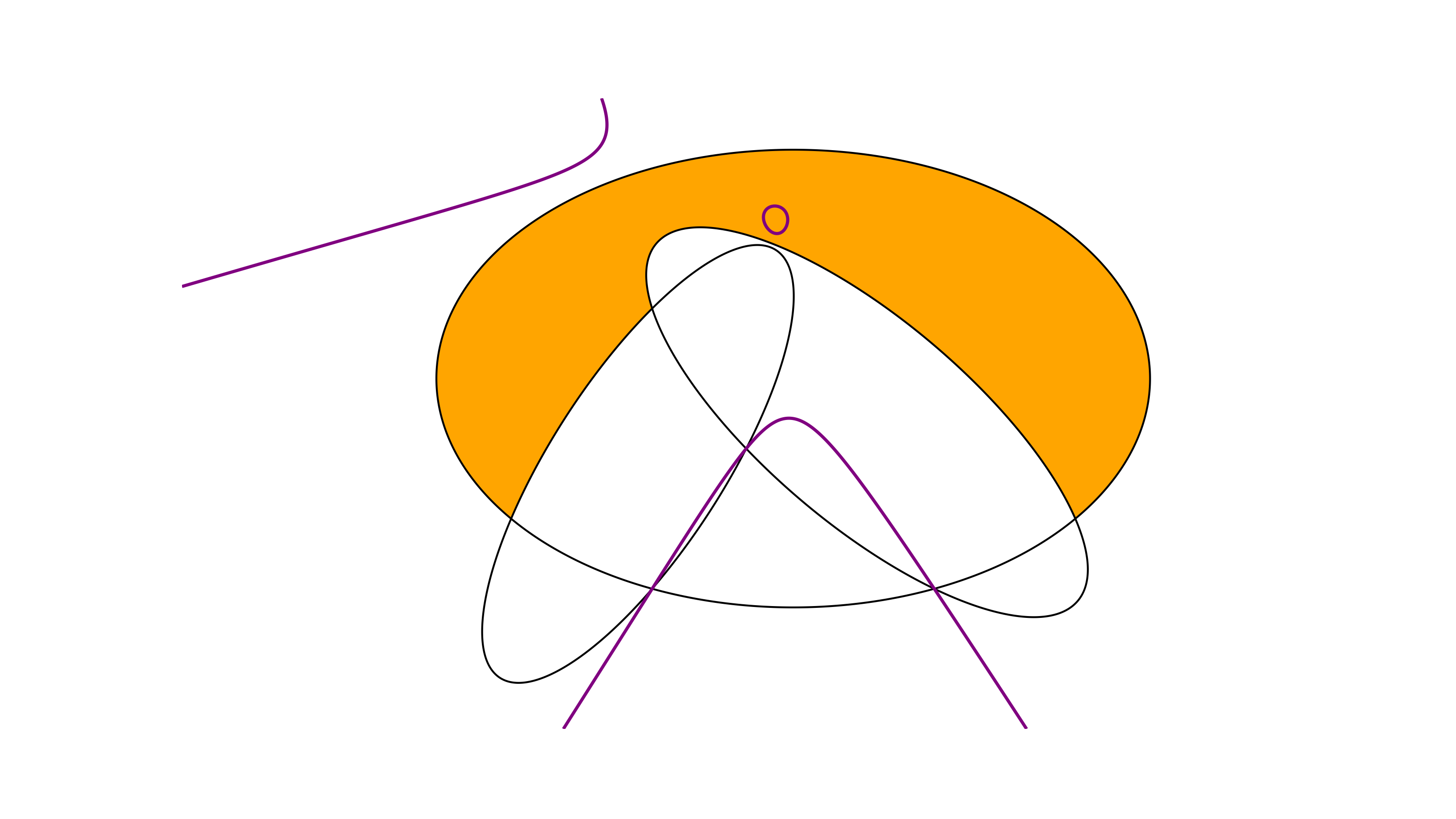}
    \caption{Counterexample to Wachspress' Conjecture.}
    \label{fig:Wachspress_CounterExample}
\end{figure}

\begin{thm} \label{thm: WachspressWrong}
 Wachspress' Conjecture is false.
\end{thm}

\begin{proof}
 Take the data from \Cref{ex: counterexample}. We need to formally prove that the described polycon $\cP$ is regular and that its adjoint vanishes in a point in $\cP_{>0}$. To facilitate computations the aid of the computer algebra system \texttt{Macaulay2} (\cite{M2}) was employed. The full code can be found in \Cref{app:Code}. We write $C_i = \cV(c_i)$. Whenever $i$ and $j$ appear as indices in the same argument we assume $i \neq j$. Similarly we assume $\{i,j,k\} = \{1,2,3\}$ whenever all three labels occur in the same argument. For notational simplicity we do not care for the order of indices, i.e. $v_{ij} = v_{ji}$.
 
 We begin by proving that every pair of two conics $C_i$ and $C_j$ intersect in exactly two distinct real points and a pair of complex conjugate points, such that in particular the polycon $\cP$ is nodal. To see this, first note that besides in the vertices of $\cP$ the boundary conics intersect in the following residual points with rational coordinates:
 \begin{equation*}
 \begin{split}
     R_{12} = (6,-8) &\in C_1 \cap C_2; \\
     R_{23} = (2,-4) &\in C_2 \cap C_3; \\
     R_{31} = (0,-8) &\in C_3 \cap C_1.
 \end{split}
 \end{equation*}
 The remaining two intersection points are given by
 \begin{equation*}
 \begin{split}
     \cV(170x-357y+1560, 289y^2-2295yz+4820) &\subset C_1 \cap C_2, \\
     \cV(196x+369y-1188, 18287y^2-67780y+63360) &\subset C_2 \cap C_3, \\
     \cV(x+18y-90, 241y^2-2316yz+5568) &\subset C_3 \cap C_1
 \end{split}
 \end{equation*}
 for each pair of boundary conics. In each case, the second condition implies that the $y$-coordinates of the intersection points are complex-valued. This proves that $\cP$ is a nodal polycon.
 
 We establish regularity by parametrizing the Euclidean boundary of $\cP$. Each $C_i$ is an ellipse in our chosen affine chart and we have the following parametrizations:
 \begin{align*}
     \pi_1 \colon C_1 &\to \pp^1 \\
     (x,y) &\mapsto \left\lbrace \begin{array}{ll}
         (y+6 : -(x+3)) & \textup{if } (x,y) \neq v_{31} \\
         (10 : 9) & \textup{else};
     \end{array} \right. \\
     \pi_2 \colon C_2 &\to \pp^1 \\
     (x,y) &\mapsto \left\lbrace \begin{array}{ll}
         (y : -x) & \textup{if } (x,y) \neq v_{23} \\
         (25 : 7) & \textup{else};
     \end{array} \right. \\
     \pi_3 \colon C_3 &\to \pp^1 \\
     (x,y) &\mapsto \left\lbrace \begin{array}{ll}
         (y : -x) & \textup{if } (x,y) \neq v_{23} \\
         (4 : -3) & \textup{else}.
     \end{array} \right.
 \end{align*}

We furthermore have the following:
\begin{align*}
    \pi_1(R_{31}) &= (2 : 3), & \pi_1(R_{12}) &= (2 : 9), & \pi_1(v_{12}) &= (0 : 1), \\
    \pi_2(R_{12}) &= (4 : 3), & \pi_2(R_{23}) &= (2 : 1), & \pi_2(v_{23}) &= (2 : 3), \\
    \pi_3(R_{23}) &= (2 : 1), & \pi_3(R_{31}) &= (1 : 0), & \pi_3(v_{31}) &= (2 : -1).
\end{align*}
In particular we observe that for $i=1,2,3$ there exists a unique segment on $C_i$ with end points $v_{ij}$ and $v_{ik}$ that does not contain any real residual points. These are
\begin{equation*}
\begin{split}
    \sigma(C_1) &= \left\lbrace \pi_1^{-1}(s : 9) \mid s \in (0,10) \right\rbrace^c, \\
    \sigma(C_2) &= \left\lbrace \pi_2^{-1}(50 : t) \mid t \in (14,75) \right\rbrace^c, \\
    \sigma(C_3) &= \left\lbrace \pi_3^{-1}(4 : t) \mid t \in (-2,-3) \right\rbrace.
\end{split}
\end{equation*}
With the so-defined sides $\sigma(C_i)$ we obtain a polycon $\cP$ without residual points on its Euclidean boundary $\partial \cP$. To determine its interior we consider the semi-algebraic set
 \begin{equation*}
     S = \{ c_1 \leq 0, c_2 \geq 0, c_3 \geq 0 \}.
 \end{equation*}
 We claim that $\partial\cP \subset S$. Since there are no residual points on $\partial \cP$, for this it suffices to verify that for each $i$ there exists a point $P \in \sigma(C_i)$ that is not a vertex of $\cP$ such that $P \in S$. Our points of choice are $(0,4) \in \sigma(C_1)$, $\left(\frac{1}{5},2\right) \in \sigma(C_2)$ and $\left(-\frac{10}{7},-\frac{16}{7}\right) \in \sigma(C_3)$. Consequently all vertices of $\cP$ lie in the same connected component $S_\cP$ of $S$. Furthermore, if $C_\cP$ vanished somewhere in the interior of $S_\cP$, then there would be $i$ and $j$ such that in a neighborhood of $v_{ij}$ on $C_i(\R)$, each point would be contained in $S_\cP$. This, however, does not hold. In particular $\cP$ is a regular polycon with $\cP_{\geq0} = S_\cP$.
 
 We conclude the proof by evaluating the adjoint polynomial $\alpha$ in the two points $P = (0,4)$ and $Q = \left(\frac{2}{5}, \frac{2}{5}\right)$. They both lie in $\cP_{\geq0}$: as noted before we have $P \in \sigma(C_1)$; furthermore, by restricting each $c_i$ to $\cV(9x+y-4)$, one obtains that the line segment connecting $P$ and $Q$ is contained in $\cP_{\geq0}$. Since we have $\alpha(P) \alpha(Q) < 0$, by the intermediate value theorem we conclude that $\alpha$ has a zero within $\cP_{\geq0}$.
\end{proof}

While we have proven that \Cref{conj: Wachspress} is false in general, it is reasonable to conjecture that it might still hold for certain ``nice'' classes of polycons. One possible restriction --- especially in light of \Cref{thm: adj_recursion} below --- is to assume that, after replacing any boundary conic with a line, the resulting polycon $\cP'$ is still regular. Note that the polycon depicted in \Cref{fig:Wachspress_CounterExample} obviously does not have this property.

\begin{conj} \label{conj: refinement}
 Suppose $\cP$ is a regular real polycon, satisfying the condition that whenever one of its boundary conics (or two adjacent boundary lines) are replaced with a line, then the resulting polycon is again regular. Then Wachspress' Conjecture holds for $\cP$.
\end{conj}

\begin{rem}
 A tentative proof of \Cref{conj: refinement} has been communicated to the author via \cite{Wachspress2019ColdCase}. At its core it suggests that using \cite[Theorem 5.2]{Wachspress2016Adjoints} one can recursively restrict considerations to polycons $\cP$, for which $\cP_{\geq0}$ is convex (see also the discussion in \cite[\S11.4]{Wachspress2016Adjoints}). For this remaining case, a direct proof is suggested. However, the presented arguments should rather be taken as a road map towards a proof, leaving many steps to be clarified or smoothed out.
\end{rem}

\subsection{Adjoint Maps}

Underlying the genesis of \Cref{ex: counterexample} is a good understanding of the set of polycons giving rise to the same adjoint. We are thus looking at the fibers of the so-called \emph{adjoint map}, which we define below. Note that due to our restricted setting (having boundary components of degree at most two), our definition is slightly more general than that given in \cite[\S4]{KohnEtAl2024AdjCurves}. For $1 \leq d \leq 2$ we denote by $\cC_d$ the variety of plane curves of degree $d$. Now for a tuple of numbers $(d_1, \dots, d_n)$ with $1 \leq d_i \leq 2$, let
\begin{equation*}
    Y_{d_1, \dots, d_n}^\circ \subset \cC_{d_1} \times \dots \times \cC_{d_n} \times (\pp^2)^n
\end{equation*}
denote the set consisting of all tuples $(C_1, \dots, C_n, v_{12}, \dots, v_{n1})$ such that
\begin{enumerate}
    \item $C_i$ and $C_j$ intersect transversally for $i \neq j$,
    \item $v_{i,i+1} \in C_i \cap C_{i+1}$,
    \item $v_{i,i+1} \not \in C_j$ for $j \not\in \{i,i+1\}$.
\end{enumerate}
The first condition implies that $C = \bigcup_i C_i$ is reduced and that both $C_i$ and $C_{i+1}$ are smooth in $v_{i,i+1}$. The Zariski closure of $Y_{d_1, \dots, d_n}^\circ$ in $\cC_{d_1} \times \dots \times \cC_{d_n} \times (\pp^2)^n$ is denoted by $Y_{d_1, \dots, d_n}$. Writing $d = d_1 + \dots + d_n$, the adjoint map $\alpha_{d_1, \dots, d_n}$ is the rational map
\begin{equation*}
    \alpha_{d_1, \dots, d_n} \colon Y_{d_1, \dots, d_n} \dashrightarrow \pp(k[x_0,x_1,x_2]_{d-3}),
\end{equation*}
which maps a polycon with boundary components $(C_i)_{i=1}^n$ satisfying $\deg(C_i) = d_i$ and with designated explicit vertices $(v_{i,i+1})_{i=1}^n$ to its unique adjoint.

One question regarding adjoint maps concerns whether they are finite and --- if so --- what their degree is. In \cite{KohnEtAl2024AdjCurves} the authors determined all cases in which the adjoint map is dominant and generically finite (but note the differences to our definition of the adjoint map). We are, however, interested in the adjoint map $\alpha_{2,2,2}$, which assigns to a polycon bounded by three conics its cubic adjoint and has infinite fibers. The proof of the following result has been communicated by Mario Kummer. A second proof is given in \Cref{cor:Yirred}.

\begin{lem}
 The variety $Y_{2,2,2}$ is irreducible.
\end{lem}

\begin{proof}
 Consider three pairwise distinct points $v_{31},v_{12},v_{23}$ in $\pp^2$ and the variety $X \subset \pp(k[x_0,x_1,x_2]_2)^3$ comprising all triples of conics $(C_1,C_2,C_3)$ with $v_{ij} \in C_i \cap C_j$ for all $i,j$.
 
 The condition that $C_i$ and $C_j$ intersect transversally for all $i,j$ and that $v_{ij} \not \in C_k$ for $k \neq i,j$ defines a non-empty open subvariety of $X$. Since $\operatorname{GL}_3(k)$ acts 3-transitively on $\pp^2$ we obtain a dominant rational map
 \begin{equation*}
     \pp^2 \times \pp^2 \times \pp^2 \times \operatorname{GL}_3(k) \dashrightarrow Y_{2,2,2}^\circ,
 \end{equation*}
 and since the source of this rational map is irreducible, so must be its target.
\end{proof}

In \Cref{sec:adjMap_fiber} we succeed in giving an explicit description of the fibers of $\alpha_{2,2,2}$ using the theory of determinantal representations.

\subsection{A Note on Positive Geometries} \label{sec:posGeom}

An emerging field of research is the study of so-called \emph{positive geometries} (\cite{arkani2017PosGeom}, \cite{Lam2022PosGeom}). It has been observed by \cite[\S2]{KohnEtAl2024AdjCurves} that there are important connections between real polycons (and their adjoints) and positive geometries in $\pp^2$ (and their canonical forms). While this relation is fundamental, it is not necessary for the present article. Thus we refrain from any discussion of positive geometries.

\section{Preliminaries on Determinantal Representations} \label{sec:PrelimAG}

The subsequent results rely on some basic notions of algebraic geometry, namely divisor theory and its relation to determinantal representations. We assume familiarity with the theory of divisors on curves as developed in \cite[\S II.6]{Hartshorne1977AG}. At times we also require familiarity with the notation and the basic statements of cohomology of coherent sheaves on curves (e.g. the Riemann--Roch Theorem and Serre duality). By a curve we always mean a projective plane curve.

Given a smooth curve $C$ and any other curve $C'$ not containing $C$, we denote by $C'.C$ the natural divisor on $C$ induced by the intersection $C'\cap C$ (counted with multiplicities). By $H$ we denote the hyperplane divisor, i.e. the divisor on $C$ cut out by a projective line. Given a Cartier divisor $D$ on $C$, by $\cO_C(D)$ we denote its associated line bundle.

The main goal of this section is to recall the relation between determinantal representations of a curve $C$ and certain divisors/sheaves on $C$. This relation has been discussed extensively throughout the last century. Modern accounts for both smooth and singular curves (and hypersurfaces) are given by \cite{Vinnikov1989LDR}, \cite{Beauville2000DetHypSurf}, \cite{KernerVinnikov2012singularDetRep}, and \cite[\S4]{Dolgachev2012CAG}. Unless stated otherwise $C = \cV(f)$ is a smooth projective plane curve of degree $d$ over the field $k \in \{\R, \C\}$.

\begin{defin}[Contact curve]
 A curve $C'$ is called a \emph{contact curve} to $C$, if $C'.C = 2D$ for some divisor $D$ on $C$. We call $D$ the \emph{contact divisor} of $C'$ on $C$.
\end{defin}

\begin{defin}[Determinantal representation]
 A \emph{linear determinantal representation} (LDR) of $C$ is a matrix $M \in M_d(k[x_0,x_1,x_2]_1)$ such that $\det(M) = \gamma f$ for some $\gamma \in k^\times$. We say that $M$ is a \emph{symmetric} LDR, if $M = M^\top$. We say that two symmetric LDRs $M, M'$ of $C$ are \emph{equivalent}, if there exists $T \in \operatorname{GL}_d(k)$ such that
 \begin{equation*}
     T M T^\top = \pm M'.
 \end{equation*}
\end{defin}

\begin{thm}[{\cite[Theorem 5]{Vinnikov1989LDR}}] \label{thm:contDiv=LDR}
 For a smooth curve $C$ there is a bijection between equivalence classes of symmetric LDRs of $C$ and divisor classes $[D]$ on $C$ satisfying $2D \sim (d-1)H$ and $\ell(D-H) = 0$.
\end{thm}

The correspondence in \Cref{thm:contDiv=LDR} can be made explicit. A classical result states that if $M$ is a symmetric LDR of $C$, then the elements of  
\begin{equation*}
    \cC_M = \left\lbrace u^tM^{\textup{adj}}u \mid u \in k^d \right\rbrace
\end{equation*}
define contact curves of degree $d-1$ to $C$ (see \cite[Theorem 5]{Vinnikov1989LDR}, \cite[\S 4.1.4]{Dolgachev2012CAG}). Conversely, given a single contact curve $C' = \cV(a)$, one can construct the adjugate matrix $M^{\textup{adj}}$ of a symmetric LDR $M$ of $C$ following Dixon's Algorithm below. This algorithm was first attributed to \cite{Dixon1902Constr}.

\begin{prop}[Dixon's Algorithm, {\cite[\S 4.1]{Dolgachev2012CAG}}] \label{prop: Dixon}
 Let $C$ be a smooth plane projective curve of degree $d$ and let $D$ be a divisor on $C$ such that $2D$ is the intersection divisor of $C$ with a curve $C' = \cV(a_{11})$ of degree $d-1$. Furthermore assume
 \begin{equation*}
     \ell(D-H) = 0.
 \end{equation*}
 Then the following algorithm determines a symmetric LDR of $C$.
 \begin{enumerate}
     \item Extend $a_{11}$ to a linearly independent tuple
     \begin{equation*}
         a_{11}, a_{12}, \dots, a_{1d}
     \end{equation*}
     in $\cL((d-1)H - D)$, the space of degree $(d-1)$ polynomials vanishing in $D$.
     
     \item Using Max Noether's Fundamental Theorem (\cite[\S 5.5]{Fulton2008Curves}), for $2 \leq i \leq j \leq d$ find $a_{ij}$ satisfying
     \begin{equation*}
        a_{11}a_{ij} - a_{1i}a_{1j} \in (f).
     \end{equation*}
     
     \item Write $M^{\textup{adj}} = (a_{ij})_{i,j}$, completing the matrix by symmetry.
     
     \item Obtain the symmetric linear determinantal representation
     \begin{equation*}
        M = f^{-d+2}(M^{\textup{adj}})^{\textup{adj}}.
    \end{equation*}
 \end{enumerate}
\end{prop}

The equivalence class of $M$ so constructed is independent of all choices made in this algorithm. In particular, if $\widetilde{C} = \cV(\widetilde{a})$ is a curve such that $\widetilde{C}.C \sim C'.C$, then there exists a change of basis $T \in \operatorname{GL}_d(k)$ such that $\widetilde{a}$ is the $(1,1)$-entry of $TM^{\textup{adj}}T^\top$.

\begin{rem}
 The adjoint of a polycon need not be smooth. This, however, does not pose a problem. Max Noether's Fundamental Theorem, on which Dixon's Algorithm relies, is applicable as long as the contact points of curves occurring in the algorithm are smooth points (\cite[\S 5.5, Proposition 1]{Fulton2008Curves}). For us this is guaranteed by \Cref{lem:AdjOffBoundary}. The only problem to be addressed is whether the determinant of the supposed determinantal representation is zero or not. For completeness we address this problem in \Cref{prop:singDixon} below, which is a special case of \cite[Proposition 5.4.]{KernerVinnikov2012singularDetRep}.
\end{rem}

\begin{prop}[Dixon's Algorithm for Singular Curves] \label{prop:singDixon}
 Let $C$ be a plane projective curve of degree $d$ and let $D$ be a divisor supported in its regular locus such that $2D$ is the intersection divisor of $C$ with a curve $C'$ of degree $d-1$. Furthermore assume
 \begin{equation*}
     h^0(C, \cO_C(D-H)) = 0.
 \end{equation*}
 Then there exists an associated symmetric LDR of $C$. Its equivalence class is independent of the divisor class of $D$.
\end{prop}

\begin{proof}
 We want to establish the conditions for the sheaf $\cO_C(D)$ laid out in \cite[Proposition 5.4]{KernerVinnikov2012singularDetRep}, which are the following (compare this also with \cite[Proposition 1.11, Theorem B]{Beauville2000DetHypSurf}):
 \begin{enumerate}
     \item $\cO_C(D) \cong \cO_C(-D) \otimes \cO_C(d-1)$.
     \item $h^0(C, \cO_C(D-H)) = 0 = h^1(C, \cO_C(D-H))$.
 \end{enumerate}
 Since $D$ is supported in the smoooth locus of $C$, the sheaf $\cO_C(-D)$ is a line bundle by \cite[Proposition 6.13.(a)]{Hartshorne1977AG}. Then we have $2D \sim (d-1)H$, which proves the first condition. For the second condition we use the dualizing sheaf $\omega_C \cong \cO_C((d-3)H)$ on $C$ and Serre duality (\cite[\S2.1.3]{KernerVinnikov2012singularDetRep}), by which we obtain
 \begin{equation*}
     h^1(C, \cO_C(D - H)) = h^0(C, \cO_C((d-3)H-(D-H)) = h^0(C, \cO_C(D-H)).
 \end{equation*}
 This proves the second condition by our hypotheses.
\end{proof}

\begin{rem}
 An explicit algorithm to construct a symmetric LDR from the divisor $D$ in the singular case is given in the proof of \cite[Theorem 4.3]{KernerVinnikov2012singularDetRep}. It is an immediate generalization of the algorithm in \Cref{prop: Dixon}.
\end{rem}

\begin{rem} \label{rem:DixonVanishing}
 For later reference we note that in Dixon's Algorithm we obtain divisors $D_i$ such that $A_{11}.C = 2D$ and $A_{1i}.C = D+D_i$ for some effective divisors $D_i$. Then $A_{ij} = D_i + D_j$ for $2 \leq i,j \leq d$.
\end{rem}

\section{Results from Algebraic Geometry} \label{sec:mainResult}

We now present some additional results of this paper, which are based on the concepts introduced in \Cref{sec:PrelimAG}. First, we establish that for the adjoint curve $A$ of a polycon $\cP$ there is a natural family of mutually linearly equivalent contact divisors, which give rise to determinantal representations. We show that the corresponding contact curves are the adjoint curves of the polycons, which are obtained by replacing a conic boundary component of $\cP$ with a line.

The second result is an explicit description of the fibers of the adjoint map $\alpha_{2,2,2}$. This tells us how $\cP$ may be deformed while preserving its adjoint curve $A$. This is synonymous to applying a symmetric change of basis to a determinantal representation $M$ of $A$. This allows us to conclude that the variety $Y_{2,2,2}$ is irreducible and that the adjoint curve of a generic polycon bounded by three conics is smooth.

\subsection{Contact Curves to the Adjoint} \label{sec:contactResults}

We begin by recalling some notation and by making some general assumptions for this section. Let $\cP$ be a nodal polycon. This will be a standing assumption for this section as we rely on the regularity part of \Cref{lem:AdjOffBoundary}. Also note that this is generically satisfied.

We denote the boundary components of $\cP$ by $(C_i)_{i=1}^n$ with $C_i = \cV(c_i)$ and its vertices by $(v_{i,i+1})_{i=1}^n$. We write $d = \deg(C_\cP)$. After possibly relabeling we assume that $\deg(C_1) = 2$. By $A = \cV(\alpha)$ we denote the adjoint curve of $\cP$ and note that $\deg(A) = d-3$. Finally we let $\cP'$ be the polycon obtained by replacing $C_1$ with the unique line $L = \cV(l)$ containing $v_{n1}$ and $v_{12}$. Its adjoint curve is denoted by $A' = \cV(\alpha')$.

\begin{rem}
 Even though we assume $\cP$ to be a regular polycon, we do not impose the same restriction on $\cP'$ (see \Cref{fig:ProblematicReduction} for a visualization).
\end{rem}

\begin{figure}[!ht]
     \centering
     \includegraphics[width=0.48\textwidth, trim={1cm 3.5cm 2cm 3cm}, clip]{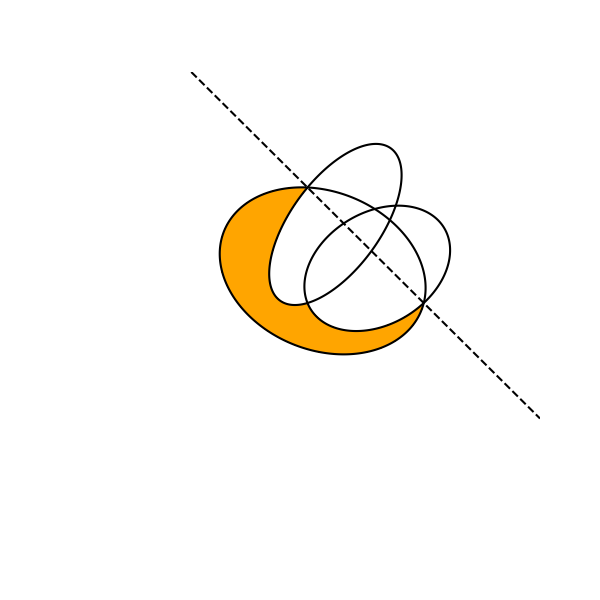}
     \includegraphics[width=0.48\textwidth, trim={1cm 3.5cm 2cm 3cm}, clip]{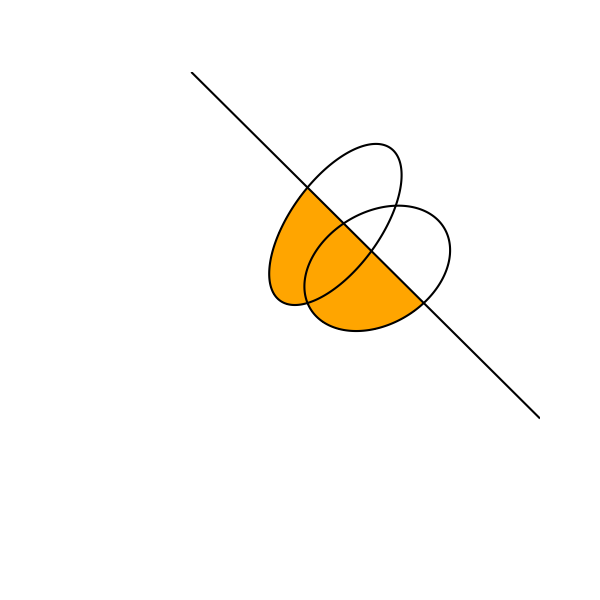}
     \caption{A regular degree six polycon (left), for which replacing one of the boundary conics with a line results in a degree five polycon that is not regular (right).}
     \label{fig:ProblematicReduction}
\end{figure}

We prove that the two curves $A$ and $A'$ are contact curves. This lets us relate adjoints and determinantal representations, and by \Cref{thm:contDiv=LDR} this can be expressed using certain divisors on $A$. Our result is a stronger version of the following proposition, which requires $A$ to be smooth and does not explicitly identify the contact curve as $A'$.

\begin{prop}[{\cite[Proposition 2.2]{AgostiniPlaumannSinnWesner2024AdjPolypol}}]\label{prop:ThetaOnAdj}
 Let $\cP$ be a nodal polycon whose adjoint curve $A$ is smooth. Then the divisor $T = R(\cP) - 3H$ on $A$ satisfies $2T \sim (d-3)H$ and $\ell(T) = 0$.
\end{prop}

While we will prove this proposition by our own means, it is useful to draw some immediate conclusions from it. Specifically, effective representatives of the divisor class $[R(\cP) - 2H] = [T+H]$ correspond to contact curves to $A$ of degree $\deg(A)-1$, and their contact divisors satisfy the conditions of \Cref{thm:contDiv=LDR}. Hence they correspond to symmetric LDRs. We can immediately observe that the intersection $R(\cP) \cap R(\cP')$ is such an effective representative: indeed, by B\'{e}zout's theorem there are exactly
\begin{equation*}
    2(d-2) - 2 = 2(d-3)
\end{equation*}
residual points of $\cP$, in which $C_1$ vanishes. The remaining residual points (those residual for both $\cP$ and $\cP'$) form a set of cardinality
\begin{equation*}
    |R(\cP) \cap R(\cP')| = |R(\cP)| - 2(d-3) = \frac{1}{2}(d-1)(d-2) - 1 - 2(d-3) = \frac{1}{2}(d-3)(d-4).
\end{equation*}
Thus, there exists a unique curve of degree $d-4$ intersecting $A$ with multiplicity two in each of the points in $R(\cP) \cap R(\cP')$. By the definition of $R(\cP')$ it is also immediate that $A'$ vanishes on $A$ in at least this set of points. However, it is not immediately clear that the intersection multiplicity is always even.

In the following we present two separate results: first, we show that $A$ and $A'$ are indeed contact curves (see \Cref{fig:AdjointTangent}). Second we prove that the contact divisor $D$ of $A'$ on $A$ satisfies $h^0(A,\cO_A(D-H)) = 0$. This in particular gives an independent proof of \Cref{prop:ThetaOnAdj}.

\begin{figure}[!ht]
     \centering
     \includegraphics[width=0.7\textwidth, trim={0cm 6cm 0cm 4cm}, clip]{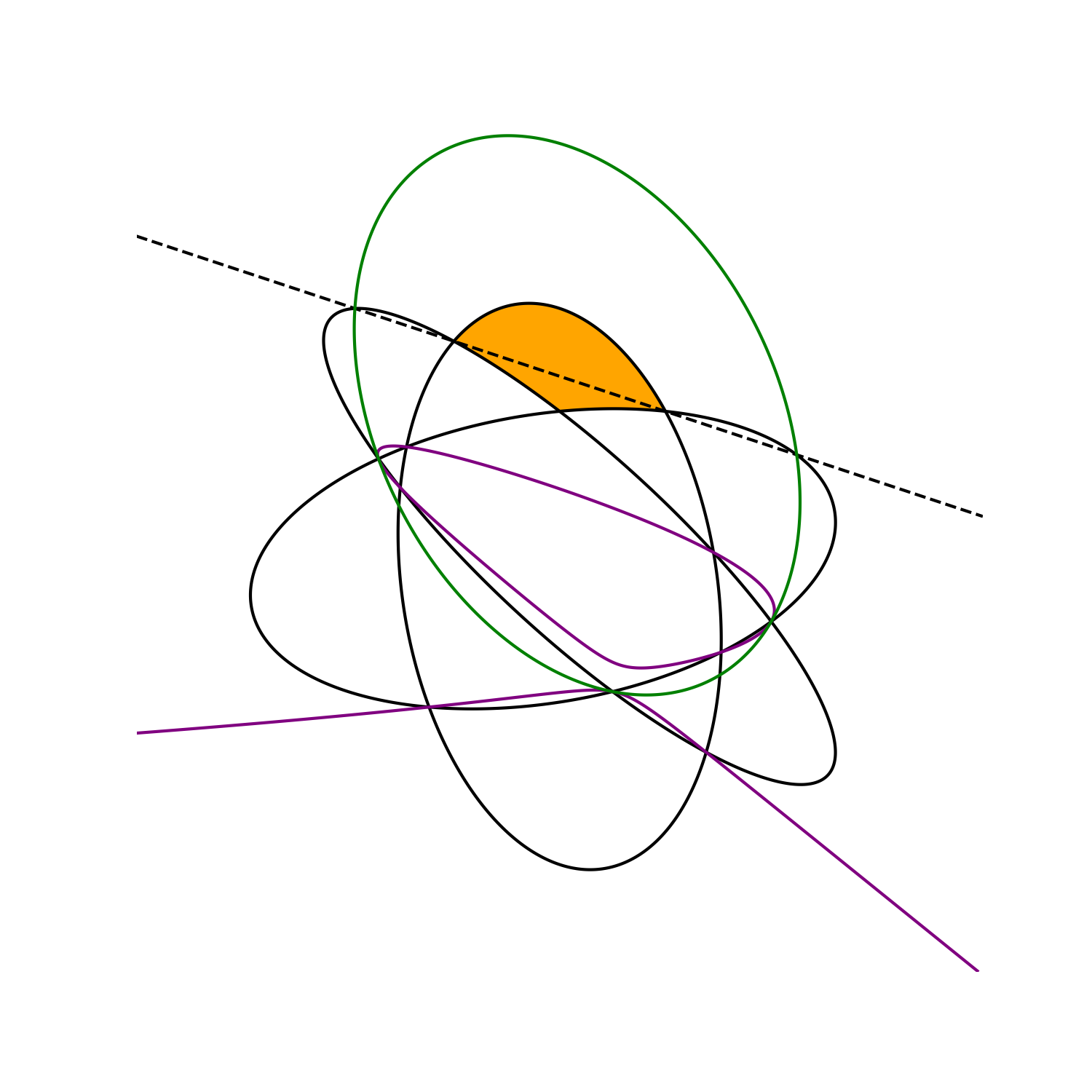}
     \caption{A polycon $\cP$ (orange) with its adjoint (purple). The green conic is the adjoint of the polycon obtained by replacing the ``inflated'' side of $\cP$ with the dashed line.}
     \label{fig:AdjointTangent}
\end{figure}

\begin{thm} \label{thm: adj_recursion}
 The curves $A$ and $A'$ are contact curves. More precisely,
 \begin{equation*}
     A'.A = 2(R(\cP) \cap R(\cP')).
 \end{equation*}
\end{thm}

\begin{proof}
 Note that it makes sense to define the divisor $A'.A$ in this fashion, since by assumption $\cP$ is nodal and hence by \Cref{lem:AdjOffBoundary} $R(\cP)$ is contained in the smooth locus of $A$. The main technical tool that we use in this proof is that adjoints behave additively in ``triangulations'' of polycons.
 
 We thus define two suitable polycons, into which we intend to partition $\cP$. The first, $\cP_0$ with adjoint $A_0 = \cV(\alpha_0)$, is the circle segment bounded by $\sigma(C_1)$ and $L$. The second is the polycon $\cP'$ with adjoint $A' = \cV(\alpha')$. The degrees of their respective adjoints are
 \begin{equation*}
     \deg(\alpha_0) = 0, \qquad \deg(\alpha') = d - 4.
 \end{equation*}
 Now by \cite[Theorem 5.1]{Wachspress2016Adjoints} there exist real numbers $b, b', b_0$, not all zero, such that
 \begin{equation*}
     b\alpha l = b'\alpha'c_1 + b_0\alpha_0\prod_{i\geq2}c_i.
 \end{equation*}
 We proceed to show that $\alpha$ and $\alpha'$ have the same tangent direction in every point $R \in R(\cP) \cap R(\cP')$. After choosing an affine chart containing $R$ and after a linear change of coordinates we may assume $R = (0,0)$ and it suffices to prove that the linear terms of $\alpha$ and $\alpha'$ are scalar multiples. Since $\prod_{i\geq2} c_i$ vanishes to order two in $R$, it does not have a linear term. This lets us reduce considerations to
 \begin{equation} \label{eq:TangentCondition}
     b\alpha l = b'\alpha'c_1 \mod (x,y)^2.
 \end{equation}
 We make a case distinction based on the possibility that $R \in L$.
 \begin{enumerate}
     \item $R \not \in L$: in this case neither $c_1$ nor $l$ vanish in $R$, hence they are invertible in $\R[x,y]_{(x,y)}$. Localizing \Cref{eq:TangentCondition} at the maximal ideal $(x,y)$ thus yields
     \begin{equation*}
         b\alpha \equiv b'\alpha' \mod (x,y)^2.
     \end{equation*}

     \item $R \in L$: in this case, $A'$ has a singular point in $R$ (\Cref{rem: AdjUnique_pathologies}). In particular, $\alpha'$ does not have a linear term, such that the claim is vacuously satisfied. \qedhere
 \end{enumerate}
\end{proof}

\begin{rem}
 \Cref{thm: adj_recursion} generalizes \cite[Proposition 5.4]{BrueserKummerPavlov2025Adjoints}, which proves the same statement for convex polygons.
\end{rem}

\begin{rem}
 We briefly mentioned in \Cref{sec:posGeom} that real polycons are related to positive geometries. Our proof of \Cref{thm: adj_recursion} translates line by line to the latter setting with only minor adaptations, using that canonical forms are additive in triangulations of positive geometries (\cite[\S3]{arkani2017PosGeom}, \cite[Theorem 11]{Lam2022PosGeom}, \cite[Proposition 2.18]{KohnEtAl2024AdjCurves}).
\end{rem}

\Cref{thm: adj_recursion} holds even if $A$ is singular, due to the regularity part of \Cref{lem:AdjOffBoundary}. We now want to prove that the contact divisor between $A'$ and $A$ satisfies the condition laid out in \Cref{prop:singDixon}.

\begin{lem}
 Let $\cP$ be a nodal polycon. Write $D = R(\cP) \cap R(\cP')$ (with $\cP'$ as in \Cref{thm: adj_recursion}). Then the associated sheaf of $D$ on $A$ satisfies the conditions of \Cref{prop:singDixon}.
\end{lem}

\begin{proof}
 By \Cref{lem:AdjOffBoundary} and \Cref{thm: adj_recursion} the set $D = R(\cP) \cap R(\cP')$ of points where $A'$ intersects $A$ are contained in the smooth locus of $A$. We thus interpret $D$ as a divisor, namely the contact divisor of $A'$ on $A$. Then $2D \sim (d-4)H$ and we claim that
 \begin{equation*}
     h^0(A, \cO_A(D-H)) = 0.
 \end{equation*}
 Note that from $D \sim (d-4)H - D$ we obtain
 \begin{equation*}
     h^0(A, \cO_A(D-H)) = h^0(A, \cO_A((d-5)H-D)),
 \end{equation*}
 such that it suffices to prove
 \begin{equation*}
     h^0(A, \cO_A((d-5)H - D)) = 0.
 \end{equation*}
 Assume the contrary, i.e. there exists a curve $B$ of degree $d-5$ vanishing in $D$ on $A$. Then the union $B \cup L$ is an adjoint of $\cP'$ in the sense of \Cref{def: adjoint}. But this contradicts \Cref{prop:uniqueAdj}, which states that the adjoint does not contain any of the boundary components of its defining polycon.
\end{proof}

We now turn to the second result of this section, which relates the fibers of $\alpha_{2,2,2}$ to symmetric LDRs of adjoints.

\subsection{Fibers of the Adjoint Map $\alpha_{2,2,2}$} \label{sec:adjMap_fiber}

We will explicitly construct fibers of $\alpha_{2,2,2}$ from determinantal representations. Unless stated otherwise, $A$ is a plane projective curve of degree three over $k \in \{\R,\C\}$. We begin with an observation that is an immediate consequence of the previous section.

\begin{lem} \label{lem:LDRFromPolycon}
 Let $\cP$ be a nodal polycon with boundary components $C_i$ and vertices $v_{ij}$. Let $A$ be its adjoint. Then there exists a symmetric LDR $M$ of $A$ such that
 \begin{equation*}
     M^{\textup{adj}} = \begin{pmatrix}
         \alpha_1 & c_3 & c_2 \\
         c_3 & \alpha_2 & c_1 \\
         c_2 & c_1 & \alpha_3
     \end{pmatrix}
 \end{equation*}
 with $C_i = \cV(c_i)$. Furthermore $A_i = \cV(\alpha_i)$ is the adjoint curve of the polycon obtained from $\cP$ by replacing $C_i$ with a line.
\end{lem}

\begin{proof}
 This is an immediate consequence of Dixon's Algorithm (\Cref{prop:singDixon}) and \Cref{rem:DixonVanishing}
\end{proof}

Conversely, given a matrix $M$ with linear entries and $A = \cV(\det(M))$, one can immediately read off a polycon with adjoint $A$ (provided $M$ satisfies some mild conditions).

\begin{prop} \label{prop:PolyconFromLDR}
 Let $M$ be a matrix of linear forms with nonzero determinant and let
 \begin{equation*}
     M^{\textup{adj}} = \begin{pmatrix}
         \alpha_1 & c_3 & c_2 \\
         c_3 & \alpha_2 & c_1 \\
         c_2 & c_1 & \alpha_3
     \end{pmatrix}.
 \end{equation*}
 Write $A = \cV(\det(M)), A_i = \cV(\alpha_i), C_i = \cV(c_i)$, and assume the following two conditions:
 \begin{enumerate}
     \item $\rk(M(P)) \geq 1$ for all $P \in \pp^2$.
     \item $\rk(M(P)) = 1$ if and only if $P \in C_1 \cap C_2 \cap C_3$.
 \end{enumerate}
 Then for $\{i,j,k\} = \{1,2,3\}$ there exists a unique point $v_{ij} \in C_i\cap C_j$ (if necessary counted with multiplicities) such that $v_{ij} \not \in A_k$. In particular, $A$ is the adjoint curve of the polycon with boundary curves $C_i$ and vertices $v_{ij}$.
\end{prop}

\begin{proof}
 We deal with the case that $A$ is smooth separately to give an idea for the involved arguments. By \Cref{rem:DixonVanishing} there exist divisors $D_1, D_2, D_3$ on $A$ with
 \begin{equation*}
    A_i.A = 2D_i, \qquad C_k.A = D_i + D_j.
 \end{equation*}
 We then observe that the common intersection of three curves $A_i, C_j, C_k$ on $A$ is the (degree three) divisor $D_i$. In particular, there exists a unique point $v_{jk} \in C_j \cap C_k$, in which $\alpha_i$ does not vanish. Finally, smoothness of $A$ implies $\rk(M(P)) \geq 2$ for all $P \in \pp^2$ (\cite[p.~885]{Vinnikov1984CubicDetRep}). As a consequence, the points $v_{ij}$ are pairwise distinct: if not (say we have $v_{12} = v_{23}$), then $v_{12} \in C_1 \cap C_2 \cap C_3$, so $\rk(M(v_{12})) = 1$ by (2), an immediate contradiction. Thus from $M^{\textup{adj}}$ we may read off a well-defined polycon with adjoint $A$ as described.

 We now turn to potentially singular curves. The first additional hypothesis is there to ensure that there exists more than one point, in which two entries of $M^{\textup{adj}}$ might have a common zero. Indeed, if there were a point $P$ with $\rk(M(P)) = 0$, then every entry of $M$ vanished in $P$. As a consequence, every entry of $M^{\textup{adj}}$ would then be singular in $P$, so that the only point where any two of them have a common zero would be $P$.
 
 Now by \cite[Proposition 5.4]{KernerVinnikov2012singularDetRep}, $M$ corresponds to a certain coherent sheaf $\cF$ on $A$ satisfying $\cF \cong \cF^\vee(d-1)$ and
 \begin{equation*}
     h^0(A, \cF(-1)) = 0 = h^1(A, \cF(-1)).
 \end{equation*}
 Following the proof of \cite[Theorem 4.3]{KernerVinnikov2012singularDetRep} we have $h^0(A, \cF) = 3$, and there exists a choice of basis $(s_1,s_2,s_3)$ of $\Gamma(A, \cF)$ such that we may identify $s_i \otimes s_j = (M^{\textup{adj}})_{ij}$. In particular, since $\deg(\alpha_i) = 2$ the section $s_i \in \Gamma(A,\cF)$ has zeros in exactly three points on $A$. Again we conclude that there exists a unique point $v_{jk} \in C_j \cap C_k$, in which $\alpha_i$ does not vanish. As before, the second additional hypothesis ensures that the points $v_{ij}$ are unique and pairwise distinct, so that we obtain a well-defined polycon.
\end{proof}

\begin{rem}
 Note that if the first condition in \Cref{prop:PolyconFromLDR} is satisfied, then the second may be obtained after a change of basis: indeed, let $P \in A \cap C_1 \cap C_2 \cap C_3$ with $\rk(M(P)) = 2$. Then there exists $i$ with $P \not \in A_i$. Without loss of generality, let $i=1$. Then
 \begin{equation*}
     T M^{\textup{adj}} T^\top, \qquad T = \begin{pmatrix}
         1 & 0 & 0 \\
         \gamma & 1 & 0 \\
         0 & 0 & 1
     \end{pmatrix}
 \end{equation*}
 is the adjugate of a symmetric LDR $M'$ of $A$ and clearly $P \in A \cap C_1 \cap C_2$ while $P \not \in C_3$.
\end{rem}

As a corollary we obtain a second proof for the irreducibility of $Y_{2,2,2}$.

\begin{cor} \label{cor:Yirred}
 The variety $Y_{2,2,2}$ is irreducible.
\end{cor}

\begin{proof}
 We define the rational map
 \begin{equation*}
     \operatorname{Sym}_3(k[x_0,x_1,x_2]_1) \dashrightarrow Y_{2,2,2}^\circ
 \end{equation*}
 that assigns to a symmetric matrix $M$ of linear forms the polycon obtained through \Cref{prop:PolyconFromLDR}. By \Cref{lem:LDRFromPolycon} this rational map is dominant, so irreducibility for $Y_{2,2,2}^\circ$ follows from that of $\operatorname{Sym}_3(k[x_0,x_1,x_2])$.
\end{proof}

We finally ask, which plane cubic curves can be realized as the adjoint of a polycon. In other words, we want to establish the plane cubic curves that admit a symmetric LDR, the rank of which does not drop to $0$. Since there is a unique flavor to real polycons that is not present in the complex case, we too make a distinction between the real and the complex adjoint map.

\begin{prop} \label{prop: fiber222}
 Let $k \in \{\R,\C\}$. If $k = \C$, then every plane cubic curve admits a symmetric LDR, the rank of which does not drop to $0$. If $k = \R$, then the unique (up to isomorphism) plane cubic curve without such a representation is $\cV(x_0(x_0^2+x_1^2+x_2^2))$. Each connected component of any fiber of $\alpha_{2,2,2}$ is parametrized by an open subset of $\operatorname{GL}_3(k)/(k^\times)^3$, acting via symmetric change of basis modulo scaling.
\end{prop}

\begin{proof}
 The result on the existence of the symmetric LDRs is taken from \cite[Table 7]{Piontkowski2006singLDR}. For the statement on the fibers we note that the parametrization is given by symmetric changes of basis modulo scaling. It remains to prove that two inequivalent symmetric LDRs $M$ and $M'$ of a cubic curve will never just differ by the scaling of their entries. Otherwise the corresponding polycons would be the same. This can be verified by evaluating $M^{\textup{adj}}$ and $(M')^{\textup{adj}}$ in suitable points. Without loss of generality we assume after (symmetric) rescaling that the first row and column of $M^{\textup{adj}}$ and $(M')^{\textup{adj}}$ are equal. If now $a_{22}' = \gamma a_{22}$ for $\gamma \neq 1$, then consider a point $P \in A$, where none of the entries of $M^{\textup{adj}}$ vanish. In this point we have the vanishing $2\times2$ minors
 \begin{equation*}
     \gamma a_{11}(P)\gamma a_{22}(P) - a_{12}(P)^2 = 0 = a_{11}(P)a_{22}(P) - a_{12}(P)^2
 \end{equation*}
 and thus $\gamma = 1$. Analogous arguments work for the scaling factors of all other entries.
\end{proof}

\begin{rem}
 Also contained in \cite[Table 7]{Piontkowski2006singLDR} is the number of inequivalent symmetric LDRs of any cubic curve, and thus the number of components of the fibers of $\alpha_{2,2,2}$.
\end{rem}

\begin{ex}
 We consider the triple line $A = \mathcal{V}(x_0^3)$. There exists a unique (up to equivalence) symmetric LDR $M$ of $A$, the rank of which does not drop to $0$ (\cite[Table 7]{Piontkowski2006singLDR}):
 \begin{equation*}
     M = \begin{pmatrix}
         x_2 & x_1 & x_0 \\
         x_1 & -x_0 & 0 \\
         x_0 & 0 & 0
     \end{pmatrix}, \qquad M_\cP = M^{\textup{adj}} = \begin{pmatrix}
         0 & 0 & x_0^2 \\
         0 & -x_0^2 & x_0x_1 \\
         x_0^2& x_0x_1 & -x_1^2-x_0x_2
     \end{pmatrix}.
 \end{equation*}
 We apply the symmetric change of basis $T = \begin{pmatrix}
         1&0&1\\
         0&1&1\\
         0&0&1
     \end{pmatrix}$ to obtain
 \begin{equation*}
     T M_\cP T^\top = \begin{pmatrix}
         2x_0^2-x_1^2-x_0x_2 & x_0^2+x_0x_1-x_1^2-x_0x_2 & x_0^2-x_1^2-x_0x_2 \\
         x_0^2+x_0x_1-x_1^2-x_0x_2 & -x_0^2+2x_0x_1-x_1^2-x_0x_2 & x_0x_1-x_1^2-x_0x_2 \\
         x_0^2-x_1^2-x_0x_2 & x_0x_1-x_1^2-x_0x_2 & -x_1^2-x_0x_2
     \end{pmatrix},
 \end{equation*}
 which can then be observed to produce a polycon with adjoint $A$. There exists exactly one residual point $(0:0:1)$, and every pair of boundary conics intersects with multiplicity three in this point. The other intersection points are the vertices of the polycon.
\end{ex}

\begin{cor}
 The adjoint of a polycon bounded by three conics is generically smooth.
\end{cor}

\begin{proof}
By \Cref{prop: fiber222} the image of $\alpha_{2,2,2}$ contains all smooth cubic curves, an open subvariety of all cubic curves. The claim then follows from irreducibility of $Y_{2,2,2}$.
\end{proof}

\bibliographystyle{alpha}
\bibliography{bibliography_red}

\appendix

\section{Code Accompanying \Cref{thm: WachspressWrong}} \label{app:Code}

\begin{verbatim}
R = QQ[x,y,z];

V31 = ideal(x+3*z,y+6*z);
V12 = ideal(x-9*z,y+6*z);
V23 = ideal(x,y);

R31 = ideal(x-0*z,y+8*z);
R12 = ideal(x-6*z,y+8*z);
R23 = ideal(x-2*z,y+4*z);

P1 = ideal(x-0*z,y-4*z);
P2 = ideal(x-5*z,y-0*z);
P3 = ideal(x-3*z,y-0*z);

C1 = (mingens intersect(V31,V12,R31,R12,P1))_(0,0);
C2 = (mingens intersect(V12,V23,R12,R23,P2))_(0,0);
C3 = (mingens intersect(V23,V31,R23,R31,P3))_(0,0);

Res = intersect(ideal(C1,C2), ideal(C2,C3), ideal(C3,C1)) :
      intersect(V12,V23,V31);
adj = (mingens Res)_(0,0);

print minimalPrimes(ideal(C1,C2) : intersect(V12,R12));
print minimalPrimes(ideal(C2,C3) : intersect(V23,R23));
print minimalPrimes(ideal(C3,C1) : intersect(V31,R31));

L = (mingens(intersect(P1,ideal(5*x-2*z, 5*y-2*z))))_(0,0);

print minimalPrimes(ideal(L, C1));
print minimalPrimes(ideal(L, C2));
print minimalPrimes(ideal(L, C3));
\end{verbatim}

\section{A Note on the Genesis of \Cref{ex: counterexample}} \label{sec:SolutionGenesis}

This appendix intends to roughly motivate \Cref{ex: counterexample}. Our original intention was to \emph{prove} \Cref{conj: Wachspress} for polycons bounded by three conics as follows: if there is no immediate argument establishing the conjecture for a polycon $\cP$, then we deform the polycon while preserving its adjoint $A$ using \Cref{prop: fiber222}, until the situation is more tractable. Ideally, throughout this deformation we retain control over the residual arrangement and preserve regularity. Let $M$ be a symmetric LDR of $A$ and let
\begin{equation*}
    M_\cP = M^{\textup{adj}} = \begin{pmatrix}
         \alpha_1 & c_3 & c_2 \\
         c_3 & \alpha_2 & c_1 \\
         c_2 & c_1 & \alpha_3
    \end{pmatrix}
\end{equation*}
be the corresponding matrix encoding $\cP$ as in \Cref{lem:LDRFromPolycon}. One family of transformations that proved promising for our goals was
\begin{equation*}
    T_\gamma = \begin{pmatrix}
    1&\gamma&0\\
    0&1&0\\
    0&0&1
\end{pmatrix}
\end{equation*}
for $\gamma \in \R$. Its simple structure implies that large parts of the residual arrangement of $\cP$ are preserved during the deformation.

\begin{lem} \label{lem: resPreserv}
 Write $C_i = \cV(c_i), A_i = \cV(\alpha_i)$ and write $\cP_i$ for the polycon obtained from replacing $C_i$ with a line $L_i$. Then the following are preserved under the action of $T_\gamma$ on $\cP$:
 \begin{enumerate}
     \item $C_1, A_2$ and $A_3$.
     \item $C_1 \cap C_2$ and $C_1 \cap C_3 \cap R(\cP)$.
     \item $L_2$, i.e. the unique line through $v_{12}$ and $v_{23}$.
 \end{enumerate}
\end{lem}

\begin{proof}
 By multiplying matrices we obtain
 \begin{equation*}
     M_\cP^\gamma = T_\gamma M_{\cP}(T_\gamma)^\top = \begin{pmatrix}
         \alpha_1 + 2\gamma c_3 + \gamma^2 \alpha_2 & c_3 + \gamma \alpha_2 & c_2 + \gamma c_1 \\
         c_3 + \gamma \alpha_2 & \alpha_2 & c_1 \\
         c_2 + \gamma c_1 & c_1 & \alpha_3
     \end{pmatrix} = \begin{pmatrix}
         \alpha_1^\gamma & c_3^\gamma & c_2^\gamma \\
         c_3^\gamma & \alpha_2 & c_1 \\
         c_2^\gamma & c_1 & \alpha_3
     \end{pmatrix}.
 \end{equation*}
 From this we immediately read off the first two statements. For the third we first note that $v_{12}$ is unchanged, since $v_{12} \in C_1 \cap C_2$. Now consider the pencil of conics $c_3^\gamma = c_3 + \gamma \alpha_2$. Since $A_2$ is the adjoint of $\cP_2$, one of the pencil's base points is the unique point in $(C_3 \cap L_2) \setminus \{v_{23}\}$. Since thus two points on $L_2$ are preserved by $T_\gamma$, so is $L_2$.
\end{proof}

We conclude with informally presenting some consequences of \Cref{lem: resPreserv}. We write $\cP^\gamma$ for the polycon corresponding to $M_\cP^\gamma$. There are three properties that would ideally hold for the family of polycons $(\cP^\gamma)_{\gamma \in I}$ for some suitable interval $I \subset \R$ with $0 \in I$:
\begin{enumerate}
    \item $\cP^\gamma$ is bounded by smooth conics.
    \item $\cP^\gamma$ is regular.
    \item the ``thickness'' of $(\cP^\gamma)_{\geq0}$ converges to $0$ as $\gamma \to \gamma_0$ for some $\gamma_0 \in \overline{I}\subset \R \cup \{\pm\infty\}$.
\end{enumerate}
There are strong indications (though a rigorous proof seems tedious) that the first and the third property may be satisfied simultaneously for the right choice of $I$ (after possibly relabeling). If the second were as well, then in light of \Cref{lem:AdjOffBoundary} the adjoint could not have an oval within $\cP_{\geq0}$. \Cref{ex: counterexample}, however, implies that this is in general not the case. Still, \Cref{lem: resPreserv} gives an indication why this is the case: since $C_1$ remains entirely unchanged, if $\cP^\gamma$ fails to be regular some $\gamma$, then this first occurs due to $C_2$ and $C_3$ being tangent in some point $P \in \partial \cP^\gamma$. While this observation is not a sufficient condition for a counterexample, it still initially gave rise to \Cref{ex: counterexample}. Note that in \Cref{fig:Wachspress_CounterExample} the two ``deflated'' sides of $\cP$ are almost tangent.
\end{document}